\documentclass[a4paper,dvipdfmx,11pt]{amsart}
\usepackage{amsmath,amssymb,amsthm}
\usepackage{enumerate}
\usepackage[all]{xy}
\usepackage{color}
\usepackage[dvipdfmx]{graphicx}
\usepackage{txfonts}
\usepackage{hyperref}

\theoremstyle{definition}
\newtheorem{defi}{\raggedright Definition}[section]
\newtheorem{prop}[defi]{\raggedright Proposition}
\newtheorem{thm}[defi]{\raggedright Theorem}
\newtheorem{exa}[defi]{\raggedright Example}
\newtheorem{lem}[defi]{\raggedright Lemma}

\newtheorem{thmn}{\raggedright Theorem}

\newcommand{\mfK}{\mathsf{K}}
\newcommand{\mfL}{\mathsf{L}}
\newcommand{\mino}{\setminus \{ \varnothing \}}
\newcommand{\ve}{\varepsilon}

\title[Hom complexes and box complexes associated to $r$-graphs]{Simple $S_r$-homotopy types of \mbox{Hom complexes} and \mbox{box complexes} associated to $r$-graphs}
\author[T.Thansri]{Thorranin Thansri}
\email{s09t253@shinshu-u.ac.jp, fscitnt@ku.ac.th}
\address{Department of Mathematical Sciences, Faculty of Science, Shinshu University, Matsumoto 390-8621, Japan} 
\date{\today}

\begin{document}
\baselineskip=5mm
\abovedisplayskip=2pt
\belowdisplayskip=2pt
\maketitle

\begin{abstract}
For a pair $(H_1,H_2)$ of graphs, Lov\'{a}sz introduced a polytopal complex called the Hom complex $\text{Hom}(H_1,H_2)$, in order to estimate topological lower bounds for chromatic numbers of graphs. The definition is generalized to hypergraphs. Denoted by $K_r^r$ the complete $r$-graph on $r$ vertices. Given an $r$-graph $H$, we compare $\text{Hom}(K_r^r,H)$ with the box complex $\mathsf{B}_{\text{edge}}(H)$, invented by Alon, Frankl and Lov\'{a}sz. We verify that $\text{Hom}(K_r^r,H)$ and $\mathsf{B}_{\text{edge}}(H)$, both are equipped with right actions of the symmetric group on $r$ letters $S_r$, are of the same simple $S_r$-homotopy type.
\end{abstract}

\section{Introduction}
In this paper, we consider homotopy types of cell complexes associated to $r$-graphs which are introduced in order to solve the problem on their chromatic numbers. The idea of assigning a cell complex to graphs was due to Lov\'{a}sz in \cite{MR514625} in his proof of the Kneser's conjecture \cite{kneser}. To a graph $G$, Lov\'{a}sz assigned a simplicial complex $\mathsf{N}(G)$, called the {\it neighborhood complex}. By using its topological property, that is to say, the $k$-connectivity of $\mathsf{N}(G)$, he succeeded in discovering a new lower bound for the chromatic number of $G$.\par
In the case of hypergraphs, the first topological lower bound for the chromatic number of an $r$-graph was derived by a simplicial complex $\mathsf{B}_{\text{edge}}(G)$ called the box complex, which was invented by Alon, Frankl and Lov\'{a}sz \cite{MR857448}. It also played an important role in a proof of the Erd\H{o}s' conjecture \cite{MR0465878}, which is a generalization of Kneser's conjecture. \par
Lov\'{a}sz also introduced a polytopal complex associated to a pair $(G,H)$ of graphs, called the Hom complex $\text{Hom}(G,H)$. It is a generalization of $\mathsf{N}(H)$ in view of $\text{Hom}(K_2,H)$ and $\mathsf{N}(H)$ having the same homotopy type \cite{MR2243723}. Here $K_2$ denotes the complete graph on 2 vertices. There are also many researches 
on the homotopy type of $\text{Hom}(K_2,H)$, comparing with other simplicial complexes constructed for (hyper)graph coloring problems such as $\mathsf{B}_{\text{chain}}(G)$ by K\v{r}\'{i}\v{z} \cite{MR1081939} or $\mathsf{B}(G), \mathsf{B}_0(G)$ by Matou\v{s}ek and Ziegler \cite{MR2073516}. However, there are still no results in the case of $r$-graphs. The motivation of this research is to find an $r$-graph which generalizes the results to the case of $r$-graphs. \par
The construction of the Hom complex is also extended to hypergraphs by Kozlov in \cite{MR2383129}. We notice here that the complete $r$-graph on $r$ vertices $K_r^r$ is the only $r$-graph having one edge as $K_2$, and that both $\text{Hom}(K_r^r,H)$ and $\mathsf{B}_\text{edge}(H)$ are equipped with right actions of the symmetric group on $r$ letters $S_r$. We obtain the following result on equivariant simple homotopy types by making use of equivariant acyclic partial matchings: 
\begin{thmn}[Theorem \ref{main_thm}]
For any $r$-graph $H$, the Hom complex $\text{Hom}(K_r^r,H)$ and the box complex $\mathsf{B}_{\text{edge}}(H)$ have the same simple $S_r$-homotopy type. 
\end{thmn}

\section{Preliminaries}
In this section, we collect some definitions which are needed in our arguments. First, we write $[k]$ as the set $\{0,1,\ldots,k \}$. 

\subsection*{$r$-graphs} A {\it hypergraph} is a triple $H=(V(H),E(H),\ve_H)$ of sets $V(H), E(H)$ and a map $\ve_H : E(H)\rightarrow \coprod_{r\geq 1} \left( V(H)^r/S_r \right)$. Here $S_r$ is the symmetric group on $r$-letters acting on $V(H)^r$ by permutation. Given two hypergraphs $H_1$ and $H_2$, a {\it hypergraph homomorphism} is a pair $(f_V,F_E)$ of $f_V : V(H_1)\rightarrow V(H_2)$ and $f_E : E(H_1)\rightarrow E(H_2)$ satisfying the following commutative diagram:
\begin{equation*}
\xymatrix{
E(H_1) \ar[rr]^{\ve_{H_1}} \ar[d]_{f_E} & & \coprod_{r\geq 1} \left( V(H_1)^r/S_r \right) \ar[d]^{\tilde{f}_V} \\
E(H_2) \ar[rr]_{\ve_{H_2}} & & \coprod_{r\geq 1} \left( V(H_2)^r/S_r \right), 
}
\end{equation*}
where $\tilde{f}_V$ is the map induced by $f_V$. Then, we obtain the category $\text{\bf H}$ of hypergraphs and hypergraph homomorphisms.\par
We denote here an equivalence class $[v_0,v_1,\ldots,v_{r-1}]\in V(H)^r/S_r$ simply by $v_0 v_1\ldots v_{r-1}$. A hypergraph $H$ is {\it $r$-uniform} if $\text{Im}\,\ve_H \subset V(H)^r/S_r$. $H$ is {\it simple} if $\ve_H$ is injective. Moreover, $H$ is {\it nondegenerate} if 
\begin{equation*}
\text{Im}\,\ve_H \subset \coprod_{r\geq 1} \left\{ v_0 \cdots v_{r-1} \in V(H)^r/S_r \,\big|\,v_i \not= v_j\ \text{whenever}\ i\not=j \right\}.
\end{equation*}
For simplicity, simple nondegenerate $r$-uniform hypergraphs are called {\it $r$-graphs}. Denoted by $\text{\bf H}_r$ the full subcategory of $\text{\bf H}$ consisting of $r$-graphs. For example, the complete $r$-graph on $m$ vertices, denoted by $K_m^r$, is an $r$-graph with $|V(K_m^r)|=m$ and $\ve_{K_m^r}$ being bijective. Since the map $\ve_H$ of an $r$-graph $H$ is a bijection $E(H)\rightarrow \text{Im}\,\ve_H$, for simply, we identify $E(H)$ with $\text{Im}\,\ve_H$, and write, for example, $v_0 \ldots v_{r-1}\in E(H)$. 
\subsection*{The category $\text{{\bf C}-}G$} Let $G$ be a group. Denoted by $G^{op}$ the group whose elements are elements of $G$ and multiplication defined by $gh\ (\text{in}\ G^{op})=hg\ (\text{in}\ G)$. For an object $X$ of a category $\text{\bf C}$, a {\it right action} of $G$ on $X$ is a homomorphism $\rho:G^{op}\rightarrow \text{Hom}_{\text{\bf C}}(X,X)$. We denote by $\text{\bf C}$-$G$ the category whose objects are all pairs $(X,\rho)$ of object $X$ of $\text{\bf C}$ and a right action $\rho$. A morphism from $(X_1,\rho_1)$ to $(X_2,\rho_2)$ is a morphism $f\in \text{Hom}_{\text{\bf C}}(X_1,X_2)$ such that $f\circ \rho_1(g)=\rho_2(g)\circ f$ for any $g\in G^{op}$. We note here that, for two categories $\text{\bf C}$ and $\text{\bf D}$, a functor $F : \text{\bf C}\rightarrow \text{\bf D}$ induces a functor $F\text{-}G: \text{{\bf C}-}G\rightarrow \text{{\bf D}-}G$.   
\subsection*{Simplicial complexes and polytopal complexes}
An {\it (abstract) simplicial complex} is a pair $(V,\mfK)$ of a set $V$ and a collection $\mfK$ of subsets of $V$ closed under taking subsets. We denote a simplicial complex $(V,\mfK)$ briefly by $\mfK$ and write $V$ as $V(\mfK)$. Each elements in $\mfK$ is called a {\it simplex} or a {\it cell} of $\mfK$. If $F\in \mfK$ and $F'\subset F$, we say that $F'$ is a {\it face} of $F$, and, at the same time, $F$ is a {\it coface} of $F'$. A {\it subcomplex} of $\mfK$ is a simplicial complex $\mfK'$ such that $F\in \mfK'$ implies that $F\in \mfK$.\par
For two simplicial complex $\mfK$ and $\mfK'$, a {\it simplicial map} $f: \mfK \rightarrow \mfK'$ is a map $f : V(\mfK)\rightarrow V(\mfK')$ satisfying that $f(F)\in \mfK'$ if $F\in\mfK$. Let $\text{\bf ASC}$ denote the category of simplicial complexes and simplicial maps. In particular, an object in the category $\text{{\bf ASC}-}G$ is called a {\it simplicial $G$-complex}.\par
Let $P$ be a convex polytope. A {\it proper face} of $P$ is of the form $\text{conv}(V(P)\cap h)$, where $h$ is a hyperplane satisfying $(\text{Int}\,P)\cap h=\varnothing$ and $V(P)$ denotes the vertex set of $P$. The term ``coface" for convex polytopes is also defined analogously. Note here that the empty set is also a proper face of any polytopes.\par
A {\it polytopal complex} is a collection $\mfK$ of convex polytopes in some $\mathbb{R}^N$ satisfying that (1) every face of $P\in \mfK$ is also in $\mfK$, and (2) the intersection of $P_1,P_2\in \mfK$ is a face of both. Elements in $\mfK$ are called {\it cells} of $\mfK$. The {\it underlying space} of a polytopal complex $\mfK$ is the subspace of $\mathbb{R}^N$ defined by $|\mfK|=\bigcup_{P\in \mfK} P$. A {\it subcomplex} of $\mfK$ is a subcollection $\mfK'$ of $\mfK$ which is itself a polytopal complex.\par
For two polytopal complexes $\mfK_1$ and $\mfK_2$, a {\it polytopal map} $f : \mfK_1\rightarrow \mfK_2$ is a map $f : |\mfK_1|\rightarrow |\mfK_2|$ satisfying that the restrictions $f|_{F}$ to each $F\in \mfK_1$ is affine. Moreover, a polytopal map $f : \mfK_1\rightarrow \mfK_2$ is said to be {\it regular} if $F\in \mfK_1$ implies that $f(F)\in \mfK_2$. In this paper, we will make use of the category $\text{\bf PTC}$ consisting of polytopal complexes and polytopal maps and its subcategory $\text{\bf PTC}_{\text{reg}}$ consisting of polytopal complexes and regular polytopal maps. An object of the category $\text{{\bf PTC}-}G$ (or $\text{\bf PTC}_{\text{reg}}\text{-}G$) is called a polytopal $G$-complex.
\subsection*{Posets} 
Let {\bf Poset} denote the category of posets and poset maps (i.e. a map $f : P\rightarrow Q$ satisfying $f(x)\leq_Q f(y)$ whenever $x\leq_P y$). An object in the category $\text{{\bf Poset}-}G$ is called a $G$-poset.\par
Given a poset $P$, we call a totally ordered subset $A=\{ A_0, A_1, \ldots, A_k \}$, where each $A_i\in P$ and $A_0 <_P A_1 <_P \cdots <_P A_{k-1}$ a {\it $k$-chain} in $P$. The number $k$ is called the {\it length} of $A$, denoted by $\# A$. In this paper, elements in a chain $A$ in $P$ are written by $A_i\ (i \in [\# A])$. The {\it order complex} of $P$, denoted by $\Delta(P)$, is the simplicial complex on $P$ whose $k$-simplices are the $k$-chains in $P$. A poset map $f : P\rightarrow Q$ induces a simplicial map $\Delta(f): \Delta(P)\rightarrow \Delta(Q)$, and so $\Delta(\cdot)$ is a covariant functor $\text{\bf Poset} \rightarrow \text{\bf ASC}$. \par
The {\it face poset} of a simplicial (polytopal) complex $\mfK$, denoted by $\mathcal{F}(\mfK)$, is a poset of all nonempty cells of $\mfK$ ordered by inclusion. Each simplicial (polytopal) map $f : \mfK\rightarrow \mfK'$ induces a poset map $\mathcal{F}(f):\mathcal{F}(\mfK)\rightarrow \mathcal{F}(\mfK')$. So we obtain covariant functors  $\mathcal{F}(\cdot) : \text{\bf ASC}\rightarrow \text{\bf Poset}$ and $\mathcal{F}(\cdot) : \text{\bf PTC}\rightarrow \text{\bf Poset}$\par
For $x,y\in P$, we call $x$ {\it covers} $y$, and write $x\succ y$, if $y <_P x$ and there is no $z\in P$ such that $y <_P z <_P x$.
\section{Equivariant simple homotopy types}
Now let $\mfK$ be a simplicial or a polytopal complex. 
Maximal cells of $\mfK$ are called {\it facets}. A cell $\sigma\in\mfK$ is {\it free} if $\sigma$ is a proper face of only one facet $\varphi_\sigma\in\mfK$. A collection $\mathcal{F}$ of free cells of $\mathsf{K}$ is said to be {\it independently free} if, for any $\sigma,\ \sigma'\in \mathcal{F}$, $\sigma\not=\sigma'$ implies that there is no cell in $\mfK$ which is a coface of both $\sigma$ and $\sigma'$.\par

The {\it deletion} of a cell $F\in \mfK$, denoted by $\text{dl}_F(\mfK)$, is the subcomplex of $\mfK$ consisting of all $F'\in \mfK$ such that $F$ is not a face of $F'$. We also define the deletion $\text{dl}_S(\mfK)$ of a set $S$ of cells of $\mfK$ from $\mfK$ as the intersection of $\text{dl}_F(\mfK)$ over all $F\in S$.\par
Now we define the notion of $G$-collapsings, following Larri\'{o}n et. al. in \cite{MR2423402}. Note here that, for a simplicial (polytopal) $G$-complex $\mfK$, the orbit $\sigma G$ of a free cell $\sigma\in \mfK$ is a collection of free cells in $\mfK$. Let $\sigma$ be a free cell of $\mfK$ with $\dim\varphi_\sigma=\dim\sigma+1$. Suppose $\sigma G$ being independently free. An {\it elementary $G$-collapsing} of $\mfK$ with respect to $\sigma$ is defined as the process to obtain $\text{dl}_{\sigma G}(\mfK)$ from $\mfK$. Conversely, an {\it elementary $G$-expanding} of $\mfK$ with respect to $\sigma$ is defined to be the process to obtain $\mfK$ from $\text{dl}_{\sigma G}(\mfK)$.\par
We denote by $\mfK \searrow_G \mfK^\prime$ if there exists an elementary $G$-collapsing of $\mfK$ onto its $G$-subcomplex $\mfK^\prime$. Moreover, we say that $\mfK$ {\it $G$-collapses} onto a $G$-subcomplex $\mfK^\prime$ if there is a sequence of elementary $G$-collapsings leading from $\mfK$ to $\mfK^\prime$. Two simplicial (polytopal) $G$-complex $\mfK$ and $\mfL$ are said to have the same {\it simple $G$-homotopy type} if there is a sequence of elementary $G$-collapsings and elementary $G$-expandings leading from $\mfK$ to $\mfL$. Such a sequence is called a {\it formal $G$-deformation}.

\subsection{Simple $G$-homotopy types of subdivisions}
It is well-known on a relationship between a simplicial (polytopal) complex $\mfK$ and its barycentric subdivision $\text{sd}\,\mfK$ that they are of the same simple homotopy type. Howover, we need an equivariant version of this result in our argument.\par
Following the construction of a formal deformation by Kozlov in \cite{MR2243723}, it is useful to define an equivariant stellar subdivision of $\mfK$.

\begin{defi}
Let $\mfK$ be a simplicial $G$-complex and $\sigma$ be a simplex of $\mfK$ such that, in $\sigma G$, $g\not= g'$ implies that no simplex in $\mfK$ being a coface of both $\sigma g$ and $\sigma g'$. The {\it stellar $G$-subdivision} of $\mfK$ at the orbit $\sigma G$, denoted by $\text{sd}\,(\mfK,\sigma G)$, is the simplicial $G$-complex on $V(\mfK)\coprod \sigma G$ with the following set of simplices:
\begin{align*}
\text{sd}(\mfK,\sigma G)=&\bigcap_{g\in G} \{ F\in \mfK\,|\,\sigma g\ \text{is not a face of}\ F \}\\
&\hspace{2ex} \cup \bigcup_{g\in G} \{ F\coprod \{ \sigma g \}  \,|\,F \in \mfK,\ \sigma g\ \text{is not a face of}\ F,\ \text{and}\ \sigma g \cup F\in \mfK \}. 
\end{align*}
We can define the stellar subdivision for a polytopal $G$-complex $\mfK$ analogously by replacing elements in $\sigma G$ with their barycenters. 
\end{defi}
Making use of stellar $G$-subdivisions, we obtain our desired result:
\begin{prop}\label{sd}
Let $\mfK$ be a simplicial or polytopal $G$-complex. Then $\mfK$ and its barcycentric subdivision $\text{sd}\,\mfK$ have the same simple $G$-homotopy type. 
\end{prop}
\begin{proof}
Choose a cell $\sigma$ from each orbit such that they preserves inclusion order in $\mathcal{F}(\mfK)$ and construct a totally ordered set $L$ of these $\sigma$'s, such that $\bigcup_{\sigma \in L} \sigma G=\mathcal{F}(\mfK)$ as sets. Then a simplicial $G$-complex obtained by a sequence of stellar $G$-subdivisions of $\mfK$ at the orbits of simplices in decreasing order with respect to $L$ is isomorphic to $\text{sd}\,\mfK$. Hence, it suffices to consider a formal deformation leading from $\mfK$ to the stellar $G$-subdivision $\text{sd}(\mfK,\sigma G)$ at the orbit of the maximum cell $\sigma\in L$.\par
First, add cones over each $\text{st}_\mfK(\sigma g)$, $g\in G$. This construction implies that, for each face $\sigma'$ of $\sigma$, $\sigma' G$ is a collection of free cells which is independently free. Hence, we obtain a sequence of elementary $G$-expandings leading to cones. Here we obtain the unique facet containing $\sigma g\in \sigma G$ in each added cone. Then we obtain our desired result by taking an elementary $G$-collapsing with respect to $\sigma G$.  
\end{proof}

\section{Hom complexes}
The construction the Hom complexes was extended to hypergraphs by Kozlov \cite{MR2383129}. In this paper, however, we will consider only the one associated to a pair of $r$-graphs. 
\begin{defi}
Let $H_1,H_2$ be $r$-graphs. A map $f:V(H_1)\rightarrow 2^{V(H_2)}\mino$ is called a {\it hypergraph multihomomorphism} if every map $f_0 : V(H_1)\rightarrow V(H_2)$ satisfying $f_0(v)\in f(v)$ for any $v\in V(H_1)$ induces a hypergraph homomorphism. 
\end{defi}
For hypergraphs $H_1,H_2$, we write $P_{H_1,H_2}$ as the poset of all hypergraph multihomomorphisms ordered by $f\leq g$ if and only if $f(v)\subset g(v)$ for any $v\in V(H_1)$. The Hom complex $\text{Hom}(H_1,H_2)$ is construced from this poset as follows:
\begin{defi}
Let $H_1,H_2$ be $r$-graphs. The {\it Hom complex} is the polytopal complex
\begin{equation*}
\text{Hom}(H_1,H_2)=\left\{ \prod_{v\in V(H_1)} \Delta^{f(v)} \right\}_{f\in P_{H_1,H_2}}.
\end{equation*}
Here $\Delta^S$ denotes a simplex with the vertex set $S$.
\end{defi}

Denoted by $\text{\bf H}^i_r$ a subcategory of $\text{\bf H}_r$ consisting of $r$-graphs and injective hypergraph homomorphisms. By definition, we obtain the following commutative diagrams concerning functorial properties: 
\begin{equation*}
\xymatrix{
\text{\bf H}_r \ar[rr]^{P_{H,-}} & & \text{\bf Poset} \\
& & \text{\bf PTC} \ar[u]_{\mathcal{F}} \\
\text{\bf H}^i_r \ar @{^{(}->}[uu] \ar[rr]_{\text{Hom}(H,-)} & & \text{\bf PTC}_{\text{reg}} \ar @{^{(}->}[u]
}
\end{equation*}
\begin{equation*}
\xymatrix{
\text{\bf H}^{op}_r \ar[rr]^{P_{-,H}} \ar[rrd]_{\text{Hom}(-,H)} & & \text{\bf Poset} \\
& & \text{\bf PTC} \ar[u]_{\mathcal{F}} \\
(\text{\bf H}^i)^{op}_r \ar @{^{(}->}[uu] \ar[rr]_{\text{Hom}(-,H)} & & \text{\bf PTC}_{\text{reg}} \ar @{^{(}->}[u]
}
\end{equation*}
In particular, we obtain right $\text{Aut}(H_1)$-actions on both the poset $P_{H_1,H_2}$ and the polytopal complex $\text{Hom}(H_1,H_2)$. Furthermore, we can see that $f<g$ in $P_{H_1,H_2}$ if and only if $\prod_{v\in V(H_1)} \Delta^{f(v)}$ is a proper face of $\prod_{v\in V(H_1)} \Delta^{g(v)}$. Therefore, $\mathcal{F}(\text{Hom}(H_1,H_2))$ and $P_{H_1,H_2}$ are $\text{Aut}(H_1)$-isomorphic as posets, and $\text{sd}\,\text{Hom}(H_1,H_2)$ and $\Delta(P_{H_1,H_2})$ are $\text{Aut}(H_1)$-isomorphic as simplicial complexes. \par
\subsection{Comparison between Hom complexes and box complexes }
Let $(G,H)$ be a pair of $r$-graphs. As stated before, we are interested in homotopy type of the Hom complex $\text{Hom}(G.H)$, comparing with simplicial complexes associated to an $r$-graph $H$. We now give the definition of the box complex $\mathsf{B}_\text{edge}(H)$ invented by Alon, Frankl and Lov\'{a}sz in \cite{MR857448}:\par
Recall that the collection $\{A_j \}_{j=0}^{r-1}$ of subsets of $V(H)$ {\it generates the complete $r$-partite sub-$r$-graph in $H$} if, for any $x_j\in A_j,\ j\in [r-1]$, $x_0 x_1\cdots x_{r-1}$ is an edge of $H$. In particular, if $V(H)=\bigcup_{j=1}^r A_j$, $H$ itself is said to be the {\it complete $r$-partite $r$-graph}, denoted by $K^r_{m_0,\ldots,m_{r-1}}$ if $|A_j|=m_j,\ j\in [r-1]$.  

\begin{defi}[See \cite{MR857448}]
Let $H$ be an $r$-graph. A simplicial complex $\mathsf{B}_{\text{edge}}(H)$ is defined to be a pair $(V,\mathsf{B}_{\text{edge}}(H))$ of the vertex set $V$ consisting of all $(v_1,\ldots,v_r)\in V(H)^r$ such that $v_1\cdots v_r\in E(H)$ and the set of simplices $B_{\text{edge}}(H)$ consisting of all subsets $F\subset V$ such that $\{\text{pr}_j(F) \}_{j=1}^r$ is the collection of pairwise disjoint sets generating the complete $r$-partite sub-$r$-graph in $H$. Here $\text{pr}_j(F)$ denotes the projection of $F$ onto its $j$-th factor. 
\end{defi}

Now we consider relationships between the Hom complexes and the box complexes. As stated before, $\text{Hom}(K_2,H)$ has the same (simple) homotopy type as the neighborhood complex $\mathsf{N}(H)$ and other box complexes. In the case of $r$-graph, since $K_2$ has only one edge, we thought that the complete $r$-graph $K_r^r$, which also has only one edge, may play an important role in determining homotopy types of the Hom complexes. Thus, we now compare homotopy types between $\text{Hom}(K_r^r,H)$ and $\mathsf{B}_\text{edge}(H)$. However, we cannot do it directly because $\text{Hom}(K_r^r,H)$ is a polytopal while $\mathsf{B}_\text{edge}(H)$ is a simplicial complex. We consider their face posets and construct two maps between them as follows: 
\begin{align*}
p : \mathcal{F}(\mathsf{B}_{\text{edge}}(H)) \rightarrow P_{K_r^r,H}&;\ p(F)(j)=\text{pr}_j(F); \\
i: P_{K_r^r,H} \rightarrow \mathcal{F}(\mathsf{B}_{\text{edge}}(H))&;\ i(\varphi)=\prod_{j=1}^r \varphi(j).
\end{align*}
Notice here that both $\text{Hom}(K_r^r,H)$ and $\mathsf{B}_\text{edge}(H)$ are equipped with right $S_r$-actions. We claim that both $p$ and $i$ are $S_r$-equivariant poset maps whose composition $p\circ i$ is the identity on $P_{K_r^r,H}$. \par
Indeed, for the $S_r$-equivariance of $p$, given a simplex $S=\{(v_1^j,\ldots,v_r^j)\}_{j\in J}\in \mathcal{F}(\mathsf{B}_{\text{edge}}(H))$ and $\sigma\in S_r$, we have $S\sigma=\{(v_{\sigma(1)}^j,\ldots,v_{\sigma(r)}^j)\}_{j\in J}$. Recall that the right $S_r$-action on $P_{K_r^r,H}$ is given as, for $\sigma\in S_r$, $\sigma: P_{K_r^r,H}\rightarrow P_{K_r^r,H};\ \varphi \mapsto \varphi \circ \sigma$. Hence, for all $l\in [r]$, 
\begin{equation*}
p(S\sigma)(l)=\{ v_{\sigma(l)}^j \}_{j\in J}=p(S)(\sigma(l))=(p(S)\sigma)(l).
\end{equation*}
For the $S_r$-equivariant of $i$, given $f\in P_{K_r^r,H}$ and $\sigma\in S_r$, we have 
\begin{align*}
i(f)\sigma &=\left(\prod_{j=1}^r f(j) \right)\sigma \\
&=\{(v_1,\ldots,v_r)\,|\, v_j\in f(j),\ j\in [r] \}\sigma \\
&=\{(v_{\sigma(1)},\ldots,v_{\sigma(r)})\,|\, v_{\sigma(j)}\in f(\sigma(j)),\ j\in [r] \} \\
&=\prod_{j=1}^r f\circ \sigma(j)=i(f\sigma).
\end{align*}
The injectivity of $i$ implies that the order complex $\Delta(i(P_{K_r^r,H}))$, which can be identified with the barycentric subdivision $\text{sd}\,\text{Hom}(K_r^r,H)$, is an $S_r$-subcomplex of $\text{sd}\,\mathsf{B}_{\text{edge}}(H)$. \par
Here we remark that, in general, the composition $i\circ p$ may not be the identity, as shown in the following example.
\begin{exa}
Consider the complete $r$-partite $r$-graph $K_{1,\ldots,1,2,2}^r$ generated by the collection 
\begin{equation*}
\{ \{a_0 \},\ldots, \{a_{r-3} \}, \{b_1,b_2 \}, \{c_1,c_2 \} \}.
\end{equation*} 
For instance, taking a simplex
\begin{equation*}
F=\{(a_0,\ldots,a_{r-3},b_1,c_1),(a_0,\ldots,a_{r-3},b_2,c_2) \} \in \mathcal{F}(\mathsf{B}_{\text{edge}}(K_{1,\ldots,1,2,2}^r),
\end{equation*}
we find that 
\begin{equation*}
\text{pr}_j(F)=\begin{cases}
\{ a_j \} &\text{if}\,j\in [r-3] \\
\{b_1,b_2 \} &\text{if}\,j=r-2 \\
\{c_1,c_2 \} &\text{if}\,j=r-1.
\end{cases}
\end{equation*}
Hence, $i\circ p(F)\not=F$. With this example, we can conclude that there is an example of $r$-graph $H$ whose poset $i(P_{K_r^r,H})$ is a proper $S_r$-subposet of $\mathcal{F}(\mathsf{B}_{\text{edge}}(H))$.\par
Moreover, we can conclude that $\Delta(i(P_{K_r^r,K_{1,\ldots,1,2,2}^{r}}))$ and $\text{sd}\,B_\text{edge}(K_{1,\ldots,1,2,2}^{r})$ are not isomorphic.
\end{exa}
We also introduce an example of $r$-graph implying that $i\circ p$ being the identity, and hence, two cell complexes are $S_r$-isomorphic:
\begin{exa}
Considering the complete $r$-partite $r$-graph $K_{1,\ldots,1,n}^{r}\ (n\in \mathbb{N})$, we find that each simplex $F$ of $\mathsf{B}_{\text{edge}}(K_{1,\ldots,1,n}^r)$ can be written as the product of sets, $r-1$ sets of them having cardinality $1$. Therefore, $i\circ p=1$.
\end{exa}
Remark here that the structures of both $\text{Hom}(K_r^r,H)$ and $\mathsf{B}_{\text{edge}}(H)$, associated to an $r$-graph $H$, depend on the containment of complete $r$-partite $r$-subgraphs in $H$. If an $r$-graph $H$ containing $K_{m_1,\ldots,m_r}^{r}$ where $|\{i\,|\,m_i\geq 2 \}|\geq 2$, then it also contains the complete $r$-partite $r$-graph $K_{1,\ldots,1,2,2}^{r}$. Together with the above examples, we obtain the following criterion of determining whether the Hom complexes and the box complexes are isomorphic: 
\begin{prop}
Let $H$ be an $r$-graph. Then $\Delta(i(P_{K_r^{r},H})) \cong \text{sd}\,\mathsf{B}_\text{edge}(H)$ if and only if $H$ does not contain the complete $r$-partite sub-$r$-graph $K_{1,\ldots,1,2,2}^{r}$.
\end{prop}
\begin{exa}
Note that the complete $r$-partite $r$-graph $K_{1,\ldots,1,2,2}^{r}$ has $r+2$ vertices. Then, for the complete $r$-graph $K_{n}^{r}$, two simplicial complexes $\Delta(i(P_{K_r^{r},K_{n}^{r}}))$ and $\text{sd}\,\mathsf{B}_\text{edge}(K_n^{r})$ are isomorphic if and only if $n\leq r+1$. 
\end{exa}

\subsection{Simple $S_r$-homotopy type of $\text{Hom}(K_r^r,H)$ and $\mathsf{B}_{\text{edge}}(H)$}
Now we return to the argument of verifying that $\text{Hom}(K_r^r,H)$ and $\mathsf{B}_{\text{edge}}(H)$ have the same simple homotopy type. Our strategy is to show that 
\begin{enumerate}
\item both $\text{Hom}(K_r^r,H)$ and $\mathsf{B}_{\text{edge}}(H)$ have the same simple homotopy type with their barycentric subdivisions, and
\item $\text{sd}\,\mathsf{B}_{\text{edge}}(H)$ $S_r$-collapses onto $\text{sd}\,\text{Hom}(K_r^r,H)$.
\end{enumerate}

The statements in the first step are proved by Proposition \ref{sd}. To prove the second one, we will verify the existence of $S_r$-collapsing of $\text{sd}\,\mathsf{B}_{\text{edge}}(H)$ onto $\Delta(i(P_{K_r^r,H}))$ by making use of an equivariant acyclic partial matching. We give here its definition and its relationships between an equivariant collapsing:  

\begin{defi}
Let $G$ be a finite group and $\mathsf{K}$ be a simplicial $G$-complex. A {\it partial $G$-matching} on $\mathcal{F}(\mathsf{K})$ is a pair $(\Sigma,\mu)$ of a $G$-subset $\Sigma$ of $\mathcal{F}(\mathsf{K})$ and a $G$-equivariant injection $\mu : \Sigma \rightarrow \mathcal{F}(\mathsf{K})\setminus \Sigma$ such that $\mu(x)\succ x$ for any $x\in \Sigma$. Elements in $\mathcal{F}(\mathsf{K})\setminus (\Sigma \cup \mu(\Sigma))$ are called {\it critical}. Such a partial $G$-matching is {\it acyclic} if there is no sequence of distinct elements $x_0,x_1,\ldots,x_t \in \Sigma\ (t\geq 1)$ such that
\begin{equation*}
\mu(x_0)\succ x_1,\ \mu(x_1)\succ x_2,\ldots,\ \mu(x_{t-1})\succ x_t\ \text{and}\ \mu(x_t)\succ x_0.
\end{equation*}
\end{defi}

\begin{prop}
Let $G$ be a finite group, $\mfK$ a simplicial $G$-complex and $\mfK'$ a $G$-subcomplex of $\mfK$. Then $\mfK$ $G$-collapses onto $\mfK'$ if and only if there is an acyclic partial $G$-matching on $\mathcal{F}(\mfK)$ whose set of critical elements is just $\mathcal{F}(\mfK')$. 
\end{prop}

\begin{proof}
First, we assume that $\mathsf{K}$ $G$-collapses onto $\mathsf{K}'$. Then we have a sequence of elementary $G$-collapsings
\begin{equation*}
\mathsf{K}=\mathsf{K}_0 \searrow_G \mathsf{K}_1 \searrow_G \mathsf{K}_2 \searrow_G \cdots \searrow_G \mathsf{K}_k=\mathsf{K}';
\end{equation*}
and we can find simplices $\sigma_0,\sigma_1,\ldots,\sigma_k$ in $\mathsf{K}$ such that, for each $i\in [k]$, $\sigma_i$ is free in $\mathsf{K}_i$; $\dim \varphi_{\sigma_i}=\dim \sigma_i+1$; $\sigma_i G$ is independently free; and $\mathsf{K}_{i+1}=\mathsf{K}_i \setminus ( \sigma_i G \cup \varphi_{\sigma_i} G)$. Let $\Sigma=\coprod_{i=0}^k \sigma_i G$; and $\mu : \Sigma \rightarrow \mathcal{F}(\mathsf{K}) \setminus \Sigma$ be defined by $\mu(\sigma_i g)=\varphi_{\sigma_i} g$. Then the pair $(\Sigma,\mu)$ is an acyclic partial $G$-matching of $\mathcal{F}(\mathsf{K})$ whose set of critical elements is $\mathcal{F}(\mathsf{K}')$.\par
We state here only a proof of $\mu$ being injective: note first that, if we let $i<j$, we find that, for any $g,g'\in G$, $\varphi_{\sigma_i} g\not\in K_j$ while $\varphi_{\sigma_j} g'\in K_j$, so $\varphi_{\sigma_i}g\not=\varphi_{\sigma_j} g'$. Hence, $\mu(G\sigma_i)\cap \mu(G\sigma_j)=\varnothing$. Then, it suffices to verify the injectivity of each restriction $\mu|_{\sigma_i G }$. \par
Suppose that there exist $g,g'\in G$ such that $\mu(\sigma_i g)=\mu(\sigma_i g')$, that is, $\varphi_{\sigma_i} g=\varphi_{\sigma_i}g'$. Then, $\varphi_{\sigma_i} g$ is a simplex in $K_i$ containing both $\sigma_i g$ and $\sigma_i g'$. Since $\sigma_i G$ is independently free, we must have $\sigma_i g=\sigma_i g'$. \par

Let us prove the converse. 
Let $(\Sigma,\mu)$ be an acyclic $G$-matching on $\mathcal{F}(\mathsf{K})$ whose set of critical elements is $\mathcal{F}(\mathsf{K}')$. 
We give here an algorithm to construct $\mathsf{K}$ from its subcomplex $\mathsf{K}'$.\par
Let $Q$ be the set of  elements of $\Sigma$ already added to $\mathsf{K}'$ and $W$ the set of minimal elements in $\mathcal{F}(\mathsf{K})\setminus \mathcal{F}(\mathsf{K}')$. 
Suppose first $Q=\varnothing$. We can find $\tau\in W$ such that, for any $g\in G$, $\mu(\tau g)=\mu(\tau) g$ is the only simplex covering $\tau g$; if not, we can choose elements of $W$ contradicting the assumption that $(\Sigma,\mu)$ is acyclic. \par   
Set $\mathsf{\bar{K}}=\mathsf{K}'\cup \tau G \cup \mu(\tau) G$. This $\mathsf{\bar{K}}$ is a simplicial $G$-complex: if there were a proper face of $\tau g$ in $\mathcal{F}(\mathsf{K})\setminus \mathcal{F}(\mathsf{K}')$, then $\tau g$ cannot be minimal in $\mathcal{F}(\mathsf{K})\setminus \mathcal{F}(\mathsf{K}')$, contradicting $\tau g\in W$. Moreover, the orbit $\tau G$ is a collection of free faces which is independently free: since $\mu$ is injective and $G$-equivariant, $\tau g \not= \tau g'$ implies that $\mu(\tau) g\not= \mu(\tau) g'$, that is, no facets in $\mathcal{F}(\mathsf{\bar{K}})$ cover both $\tau g$ and $\tau g'$ if $g \not= g'$. So we can conclude that $\mathsf{\bar{K}}$ elementary $G$-collapses onto $\mathsf{K}'$. \par
Delete all elements in $\tau G$ from $W$, set $Q:=Q\cup \tau G \cup \mu(\tau) G,\ \mathsf{K}'=\mathsf{\bar{K}}$, and repeat our argument until $W=\varnothing$. 
If $W=\varnothing$, take a new $W$ of minimal elements in $\mathcal{F}(\mfK) \setminus (\mathcal{F}(\mfK')\cup Q)$ and continue our argument until $Q=\mathcal{F}(\mathsf{K}) \setminus \mathcal{F}(\mathsf{K}')=\Sigma\cup \mu(\Sigma)$; and we obtain a sequence of elementary $G$-collapsings leading from $\mathsf{K}$ to $\mathsf{K}'$.
\end{proof}
By this proposition, if one wants to verify that two simplicial $G$-complexes have the same simple homotopy types, it suffices to construct an acyclic partial $G$-matching on their face posets. Now we give a construction for our main result:
\begin{lem}
For an $r$-graph $H$, $\text{sd}\,\mathsf{B}_{\text{edge}}(H)$ $S_r$-collapses onto $\Delta(i(P_{K_r^r,H}))$.
\end{lem}
\begin{proof}
Since $\Delta(i(P_{K_r^r,H}))$ is a $S_r$-subcomplex of $\text{sd}\,B_\text{edge}(H)$, we will construct an acyclic partial $S_r$-matching on $\mathcal{F}(\text{sd}\,B_\text{edge}(H))$ whose set of critical elements is $\mathcal{F}(\Delta(i(P_{K_r^r,H})))$.  \par
Note first that, for any chain $A$ of $\text{sd}\,\mathsf{B}_{\text{edge}}(H)$, $A$ is a chain of $\Delta(i(P_{K_r^r,H}))$ if and only if $i\circ p(A_k)=A_k$ for any $k\in [\# A]$. Indeed, if $A$ is a chain of $\Delta(i(P_{K_r^r,H}))$, then we can choose $\varphi_k\in P_{K_r^r,H}$ and write $A_k=i(\varphi_k)$ for each $k\in [\# A]$. Since $p\circ i=1$, we obtain $i\circ p(A_k)=A_k$. The converse holds by the definitions of $i$ and $p$. \par
To achieve our purpose, it suffices to construct an acyclic partial $S_r$-matching which matches chains not belonging to  $\Delta(i(P_{K_2,G}))$. First, we define a subset $D$ of $ \mathcal{F}( \text{sd}\,\mathsf{B}_{\text{edge}}(H))$ by
\begin{equation*}
D=\{ F\in \mathcal{F}( \text{sd}\,\mathsf{B}_{\text{edge}}(H))\,|\,i\circ p(F_j)\not=F_j\ \text{for some}\ j\in [\# F] \}.
\end{equation*}
$D=\varnothing$ implies that $\Delta(i(P_{K_r^r,H}))$ and $\text{sd}\,\mathsf{B}_{\text{edge}}(H)$ are the same. We assume $D\not=\varnothing$. For any $F\in D$, we let $l(F)$ denote the minimal index $l$ such that $i\circ p(F_l)\not=F_l$, and $r(F)$ the maximal index $r$ such that $F_{l(F)+r}$ is included in $i\circ p(F_{l(F)})$. With these indices, we define $\Sigma_1,\Sigma_2\subset D$ as follow:
\begin{align*}
\Sigma_1&=\{F\in D\,|\,l(F)+r(F)=\# F,\ i\circ p(F_{l(F)})\in F \};\\
\Sigma_2&=\left\{F\in D\,|\,l(F)+r(F)<\# F,\ i\circ p(F_{l(F)})\cap F_{l(F)+r(F)+1}\in F \right\}.
\end{align*}

Now we define a map $\mu : \Sigma_1\cup \Sigma_2\rightarrow \mathcal{F}(\text{sd}\,B_\text{edge}(H))\setminus (\Sigma_1 \cup \Sigma_2)$ as  
\begin{equation*}
\mu(F)=\begin{cases}
F\cup \{i\circ p(F_{l(F)}) \} &\text{if}\ F\in\Sigma_1;\\
F\cup \{i\circ p(F_{l(F)})\cap F_{l(F)+r(F)+1} \} &\text{if}\ F\in\Sigma_2.
\end{cases}
\end{equation*}
We claim that the pair $(\Sigma_1\cup \Sigma_2,\mu)$ is an acyclic partial $S_r$-matching on $\mathcal{F}(\text{sd}\,\mathsf{B}_\text{edge}(H))$.\par
We first check that $\Sigma_1\cup \Sigma_2$ is an $S_r$-subset of $\mathcal{F}(\text{sd}\,\mathsf{B}_\text{edge}(H))$: 
let $F=\{F_0,F_1,\ldots,F_{\# F} \}$ be an element of $\Sigma_1\cup \Sigma_2 \subset \mathcal{F}(\text{sd}\,\mathsf{B}_\text{edge}(H))$ satisfying $F_0\subset F_1 \subset \ldots \subset F_{\# F}$ and $\sigma\in S_r$. Then $F \sigma=\{F_0 \sigma,F_1 \sigma,\ldots,F_{\# F} \sigma \}$ is a chain of $\text{sd}\,\mathsf{B}_\text{edge}(H)$. Since $F\in D$, we can take an index $j$ with $i\circ p(F_j)\not=F_j$. Then $S_r$-equivariance of $i\circ p$ implies that $i\circ p(F_j \sigma)\not=F_j\sigma$. So $F\sigma \in D$. \par
Now suppose $F\in\Sigma_1$. The condition $l(F\sigma)+r(F\sigma)=\# F$ holds because of the bijectivity of $\sigma$. Since $i\circ p(F_{l(F)})\not\in F$, $(F\sigma)_{l(F\sigma)}=F_{l(F)}\sigma$ and $i\circ p$ is $S_r$-equivariant, we have $i\circ p((F\sigma)_{l(F\sigma)})\not\in F\sigma$, and so $F\sigma\in \Sigma_1$. Next let $F\in\Sigma_2$. The condition  $l(F\sigma)+r(F\sigma)<\# F$ is obvious. The second condition comes from the following calculation:
\begin{align*}
i \circ p((F\sigma)_{l(F\sigma)})\cap (F\sigma)_{l(F \sigma)+r(F \sigma)+1}&=i \circ p(F_{l(F)}\sigma)\cap F_{l(F)+r(F)+1}\sigma \\
&=i \circ p(F_{l(F)})\sigma \cap F_{l(F)+r(F)+1}\sigma \\
&=(i\circ p(F_{l(F)}) \cap F_{l(F)+r(F)+1}) \sigma \not\in F\sigma.
\end{align*}
So $F\sigma \in \Sigma_2$. Summing up, $\Sigma_1\cup \Sigma_2$ is an $S_r$-subset.\par
Next, we must verify that $\mu$ satisfies the condition for being a partial $S_r$-matching: First we find that both $i\circ p(F_{l(F)})$ for $F\in\Sigma_1$ and $i\circ p(F_{l(F)})\cap F_{l(F)+r(F)+1}$ for $F\in\Sigma_2$ are simplices of $B_\text{edge}(H)$. Hence, $\mu(F)$ is a chain in $\text{sd}\,B_\text{edge}(H)$ with relation 
\begin{equation}\label{mu_sigma1}
F_0 \subset \ldots \subset F_{l(F)} \subset \ldots \subset \cdots \subset F_{\# F} \subset i\circ p(F_{l(F)})
\end{equation}
for $F\in \Sigma_1$, and 
\begin{align}\label{mu_sigma2}
&F_0\subset \ldots \subset F_{l(F)} \subset \ldots \subset F_{l(F)+r(F)} \subset i\circ p(F_{l(F)})\cap F_{l(F)+r(F)+1} \subset F_{l(F)+r(F)+1} \subset \ldots \subset F_{\# F}.
\end{align}
for $F\in \Sigma_2$. We can see from the relations \eqref{mu_sigma1} and \eqref{mu_sigma2} that, for any $F\in \Sigma_1\cup \Sigma_2$, $\mu(F)$ covers $F$ but is not a chain in $\Sigma_1\cup \Sigma_2$; moreover, $F_1\in \Sigma_1$ and $F_2\in \Sigma_2$ imply that $\mu(F_1)\not= \mu(F_2)$.  If we suppose that both $F_1$ and $F_2$ belong to $\Sigma_j\ (j=1,2)$ satisfying $\mu(F_1)=\mu(F_2)$, then we find that the inserted terms to obtain $\mu(F_1)$ and $\mu(F_2)$ are in the same index. This yields that $F_1=F_2$, and so $\mu$ is injective. This $\mu$ is $S_r$-equivariant because of the following calculations: if $F\in\Sigma_1$,
\begin{align*}
\mu(F\sigma)&=F\sigma \cup \{i\circ p((F\sigma )_{l(F\sigma )}) \}\\
&=F\sigma \cup \{i\circ p(F_{l(F)})\sigma \}\\
&=(F \cup \{i\circ p(F_{l(F)}) \})\sigma=\mu(F)\sigma.
\end{align*}
If $F\in \Sigma_2$, we have
\begin{align*}
\mu(F\sigma)&=F\sigma \cup \{ i\circ p( (F\sigma )_{l(F\sigma )} ) \cap (F\sigma )_{l(F\sigma )+r(F \sigma )+1 } \}  \\
&=F\sigma \cup \{ ( i \circ p(F_{l(F)})\cap F_{l(F)+r(F)+1}  )\sigma \} \\
&=(F \cup \{ ( i \circ p(F_{l(F)})\cap F_{l(F)+r(F)+1}) \} )\sigma=\mu(F)\sigma. 
\end{align*}
Finally, we find that $\Sigma_1\cup \Sigma_2 \cup \mu(\Sigma_1\cup \Sigma_2)=D$, and we can conclude that the pair $(\Sigma_1\cup \Sigma_2,\mu)$ is a partial $S_r$-matching on $\mathcal{F}(\text{sd}\,B_\text{edge}(H))$ whose set of critical elements is $\mathcal{F}(\Delta(i(P_{K_r^r,H})))$. \par
It remains to prove that the matching is acyclic: suppose that there exists a sequence of distinct elements $F^0,F^1,\ldots,F^t \in \Sigma_1 \cup \Sigma_2\ t\geq 1$ such that
\begin{equation*}
\mu(F^0)\succ F^1,\mu(F^1) \succ F^2,\ \ldots\ ,\ \mu(F^{t-1})\succ F^t\ \text{and}\ \mu(F^t) \succ F^0.
\end{equation*}

For each $j\in [t-1]$, since $\mu(F^j)$ covers both $F^j$ and $F^{j+1}$ which are distinct, we can choose a simplex $A_j\in F^j$ such that $F^{j+1}=\mu(F^j)\setminus \{A_j \}$. Similarly, $A_t \in F^t$ can be chosen such that $F^0=\mu(F^t)\setminus \{A_t \}$.\par
It is useful if we know what are $A_j,\ j\in [t]$: we claim here that
\begin{equation*}
A_j=\begin{cases}
F^j_{l(F^j)}&\ \text{if}\ F^j\in\Sigma_1; \\
F^j_{l(F^j)+r(F^j)+1}\ \text{or}\ F^j_{l(F^j)}&\ \text{if}\ F^j\in \Sigma_2.
\end{cases}
\end{equation*}
In fact, for $F^j\in \Sigma_1$, if $A_j$ were not $F^j_{l(F^j)}$, it follows from the equation \eqref{mu_sigma1} that $F^{j+1}_{l(F^{j+1})}=F^j_{l(F^j)}$, and so $i\circ p(F^{j+1}_{l(F^{j+1})}) \in F^{j+1}$; hence $F^{j+1} \not\in \Sigma_1$. Since $i\circ p(F^{j}_{l(F^j)})$ contains all simplices in $F^j$, we obtain $F^{j+1}\not\in \Sigma_2$. Therefore $F^{j+1}\not\in \Sigma_1\cup \Sigma_2$, contradicting to the assumption of $F^{j+1}$. For $F^j\in\Sigma_2$, if $A_j$ were not $F^j_{l(F^j)}$ and $F^j_{l(F^j)+r(F^j)+1}$, it follows from the equation \eqref{mu_sigma2} that  $F^{j+1}_{l(F^{j+1})}=F^j_{l(F^j)}$. So $F^{j+1}\not\in \Sigma_1$ because the simplex $F^{j+1}_{l(F^{j+1})+r(F^{j+1})+1}$ still exists. Moreover, we obtain $F^{j+1}_{l(F^{j+1})+r(F^{j+1})+1}=F^j_{l(F^j)+r(F^j)+1}$, and then $i\circ p(F^{j+1}_{l(F^{j+1})}) \cap F^{j+1}_{l(F^{j+1})+r(F^{j+1})+1} \in F^{j+1}$. Hence $F^{j+1}\not\in \Sigma_2$. Summing up, $F^{j+1}\not\in \Sigma_1\cup \Sigma_2$, which contradicts to the assumption of $F^{j+1}$. \par

We can see from the above remark on $A_j$ that, if $F^j\in\Sigma_2$, $F^{j+1}$ can be a chain in either $\Sigma_1$ or $\Sigma_2$, while, if $F^j\in \Sigma_1$,  $F^{j+1}$ can be a chain only in $\Sigma_1$ because $i\circ p(F^{j+1}_{l(F^{j+1})})$ contains $i\circ p(F^j_{l(F^j)})$, which contains all $F^j_k\ (k\in [ \# F^j ])$. Similarly, $F^t\in \Sigma_1$ implies that $F^0\in \Sigma_1$. Then we can conclude that there are three cases on a set to which the chains $F^0, \ldots, F^t$ belongs, as follows:
\begin{enumerate}[(a)]
\item All $F^0,\ldots,F^t$ belong to $\Sigma_1$;
\item All $F^0,\ldots,F^t$ belong to $\Sigma_2$;
\item There exists $j\in [t-1]$ such that $F^j\in \Sigma_2$ but $F^{j+1}\in \Sigma_1$. 
\end{enumerate}  
We can find a contradiction for the case (c) at once because the fact that $F^k \in \Sigma_1$ whenever $F^{k-1} \in \Sigma_1$ implies that $F^j\in \Sigma_1$. For the case (a), considering the number $t(F^j)$ of indices $l$ such that $F^j_{l}\not=i\circ p(F^j_l)$, we obtain a contradiction $t(F^0)<t(F^0)$. \par
For the case (b), let we denote $s(F^j)$ the number of simplices in $F^j$ not contained in $i\circ p(F^j_{l(F^j)})$.
By assumption, we have $s(F^j)\geq 1$for any $j\in [t]$. By the assumption, each $A_j$ is $F^j_{l(F^j)}$ or $F^j_{l(F^j)+r(F^j)+1}$. If $A_j=F^j_{l(F^j)}$, the fact that $i\circ p(F^{j+1}_{l(F^{j+1})}) \supset i\circ p(F^j_{l(F^j)})$ implies that $s(F^{j+1})\leq s(F^j)$. If $A_j=F^j_{l(F^j)+r(F^j)+1}$, then we have $s(F^{j+1})=s(F^j)-1<s(F^j)$. Summing up, $F^0,\ldots,F^t\in \Sigma_2$ implies the following inequalities:
\begin{equation}\label{ineq_s}
s(F^0)\leq s(F^t) \leq \cdots \leq s(F^1) \leq s(F^0).
\end{equation}
We will get a contradiction if there exists a ``less than or equal to" sign which is really the ``less than" sign. We obtain the assertion at once if there is $j\in [t]$ with $A_j=F^j_{l(F^j)+r(F^j)+1}$. \par       
Assume that $A_j=F^j_{l(F^j)}$ for all $j\in [t]$. By definition, we can choose $F^{j+1}_{l(F^{j+1})}$ and $F^{j+1}_{l(F^{j+1})+r(F^{j+1})+1}$ in each $F^j$. However, we will get a contradiction 
\begin{equation*}
F^{j+1}\ni i \circ p(F^{j+1}_{l(F^{j+1})})\cap F^{j+1}_{l(F^{j+1})+r(F^{j+1})+1}
\end{equation*}
if there exists $j\in [t]$ such that either of these conditions holds:
\begin{enumerate}[(c1)]
\item $i \circ p(F^j_{l(F^j)})\cap F^j_{l(F^j)+r(F^j)+1}=i \circ p(F^{j+1}_{l(F^{j+1})})\cap F^{j+1}_{l(F^{j+1})+r(F^{j+1})+1}$, or
\item $i \circ p(F^{j+1}_{l(F^{j+1})})\cap F^{j+1}_{l(F^{j+1})+r(F^{j+1})+1}$ is distinct from $F^j_{l(F^j)}$ and is in $F^j$.
\end{enumerate}
Then we can assume that all $j\in [t]$ do not satisfy both conditions. Suppose that $s(F^0)=s(F^1)=\cdots=s(F^t)$. We find that $F^0_{l(F^0)+r(F^0)+1}$ is the minimal simplex not included in $i\circ p(F_{l(F^j)})$ for any $j\in [t]$. Paying attention to the simplices inserted to each chain, we find by our assumption that
\begin{align}
&i\circ p(F_{l(F^0)})\cap F_{l(F^0)+r(F^0)+1} \subsetneq i\circ p(F_{l(F^1)})\cap F_{l(F^0)+r(F^0)+1} \subsetneq \cdots  \nonumber \\
&\hspace{5ex} \subsetneq i\circ p(F_{l(F^t)})\cap F_{l(F^0)+r(F^0)+1}\subsetneq F_{l(F^0)+r(F^0)+1}. \label{contra_subset}
\end{align}     
Since $F_{l(F^0)+r(F^0)+1}$ is the minimal simplex not included in $i\circ p(F_{l(F^0)})$, we obtain
\begin{equation*}
i\circ p(F_{l(F^t)})\cap F_{l(F^0)+r(F^0)+1} \subset i\circ p(F_{l(F^0)}).
\end{equation*}
Then,
\begin{equation*}
i\circ p(F_{l(F^t)})\cap F_{l(F^0)+r(F^0)+1} \subset i\circ p(F_{l(F^0)}) \cap F_{l(F^0)+r(F^0)+1}.
\end{equation*}
With \eqref{contra_subset}, we thus obtain a contradiction $i\circ p(f_{l(F^0)}) \cap f_{l(F^0)+r(F^0)+1}\subsetneq i\circ p(f_{l(F^0)}) \cap f_{l(F^0)+r(F^0)+1}$. Therefore, in \eqref{ineq_s}, there exists a ``less than or equal to" sign which is really the ``less than" sign, and so we get a contradiction $s(F^0)<s(F^0)$. 

Summing up, our argument contradicts itself if we suppose that $(\Sigma_1\cup \Sigma_2,\mu)$ is not acyclic. 
\end{proof}
We depict an $S_r$-collapsing construced by the above acyclic partial $S_r$-matching for a part of $\text{sd}\,\mathsf{B}_\text{edge}(H),\ H=K^3_{2,2,1}$ as the following figure. Here we draw a hypergraph by edge-based drawings, see \cite{drawing}.
\begin{figure}[h]
\begin{center}
\includegraphics{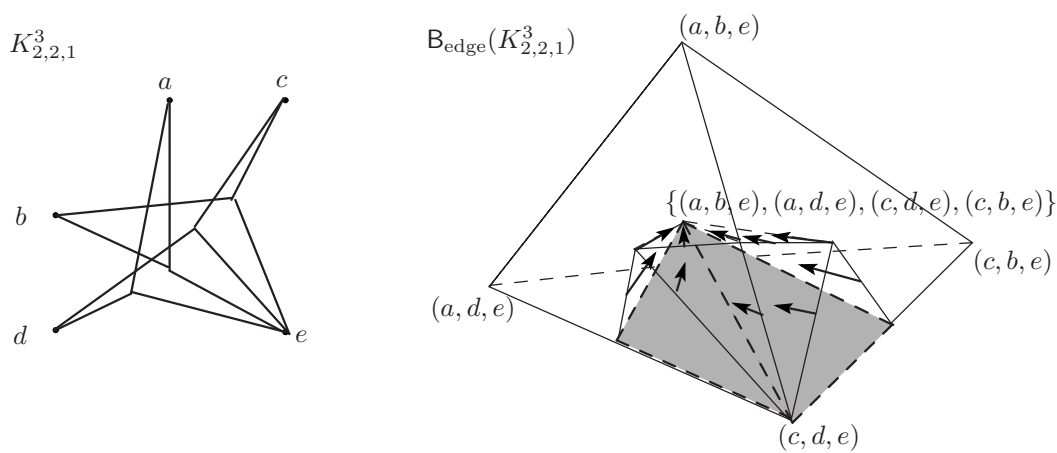}
\end{center}
\caption{$K^3_{2,2,1}$ and a part of the $S_3$-collapsing of $\text{sd}\,\mathsf{B}_\text{edge}(K^3_{2,2,1})$ onto $\Delta(i(P_{K_3^3,K^3_{2,2,1}}))$.} 
\end{figure}

We now complete our argument in all steps, obtaining a construction of a formal $S_r$-deformation between $\text{Hom}(K_r^{r},H)$ and $\mathsf{B}_{\text{edge}}(H)$. So the following conclusion holds:
\begin{thm}\label{main_thm}
For an $r$-graph $H$, the Hom complex $\text{Hom}(K_r^{r},H)$ and the box complex $\mathsf{B}_{\text{edge}}(H) $ have the same simple $S_r$-homotopy type. 
\end{thm}

\section{Acknowledgement} 
The author would like to express his gratitude to his advisor, Professor \mbox{Dai Tamaki} and Professor \mbox{Katsuhiko Kuribayashi} for many helpful suggestions and many opportunities of discussions concerning this research. The author also wishes to thank the organizers of Symposium on Homotopy Theory for giving me an opportunity to give a talk and have useful discussions concerning with this paper.   
            
\providecommand{\bysame}{\leavevmode\hbox to3em{\hrulefill}\thinspace}
\providecommand{\MR}{\relax\ifhmode\unskip\space\fi MR }
\providecommand{\MRhref}[2]{%
  \href{http://www.ams.org/mathscinet-getitem?mr=#1}{#2}
}
\providecommand{\href}[2]{#2}

\end{document}